\documentclass[a4paper]{amsart}
\usepackage{stile2017}

\title{Arithmetic properties of cubic and biquadratic theta series}
\author{Luca Ghidelli}
\date{\today}

\address{150 Louis-Pasteur Private, Office 608, Department of Mathematics and Statistics, University of Ottawa, Ottawa ON K1N 9A7, Canada}
\email{{luca.ghidelli@uottawa.ca}}
\renewcommand{\th}{\theta_\ell}
\newcommand{\e}{\epsilon}
\renewcommand{\P}{\mathcal P}

\newcommand{\mild}{\op{MildGap}}

\newcommand{\Bad}{\mcl B^{\text{bad}}}
\newcommand{\Good}{\mcl B ^{\text{good}}}

\crefname{enumi}{}{}

\begin{document}

\begin{abstract}
A cubic (resp. biquadratic) theta series is a power series whose n-th coefficient is equal to 1 if n is a perfect cube (resp. fourth power) and zero otherwise. 
%Although this is an example of lacunary series, the spacing between the nonzero coefficients is not large enough to imply transcendence of the values of the series. 
We improve on a result of Bradshaw by showing that such series is not a cubic (resp. biquadratic) algebraic number when evaluated at reciprocals of integers. 
The proof relies on a ``nested gaps technique'' for linear independence and on recent results by the author on Waring's problem for cubes and biquadrates.
\end{abstract}

\maketitle

\tableofcontents

\section{Introduction}
\label{sec:statement}

In this paper we consider numbers of the form
\[
	\th(q) = \sum_{n=0}^\infty \frac 1 {q^{n^\ell}}, 
\]
for $\ell\in\{3,4\}$ and $q>1$. 
These numbers can be thought as being values (at $z=1/q$) of cubic/biquadratic generalizations of the well-known theta series $\sum_{n=0}^\infty z^{n^2}$. 
As usual for values of transcendental series, we expect that $\th(q)$ is transcendental at algebraic inputs, possibly with some well-motivated exceptions.
Our main result is the following.
\begin{theorem}\label{thm:main}
Let $\ell\in\{3,4\}$, let $q\geq 2$ be an integer and suppose that $\th(q)$ is algebraic. 
Then $\deg \th(q)\geq \ell+1$.
\end{theorem}
%More precisely, we give a measure of linear independence between the first $\ell+1$ powers of $\th(q)$:
%\begin{theorem}\label{thm:main:measure}
%Let $\ell,q$ be as in \cref{thm:main}, let $H\great 0$ and let $L(\underline {T}) = a_0 T_0 + \ldots + a_\ell T_\ell$ be a linear form with integer coefficients $a_i$, not all zero, such that $\abs{a_i}\leq H$. Then if $\ell=3$:
%\[
%\abs{L(1,\theta_3(q),\theta_3(q)^2,\theta_3(q)^3)} > \exp \left( -\exp(C(\log H)^r)\right)
%\]
%for all $r>2$ and some $C=C(r)>0$. If $\ell=4$:
%\[
%\abs{L(1,\theta_4(q),\theta_4(q)^2,\theta_4(q)^3,\theta_4(q)^4)} > \exp \left( -\exp(C\exp (H^r))\right)
%\]
%for all $r>1$ and some $C=C(r)>0$
%\end{theorem}

%Clearly \cref{thm:main} follows from \cref{thm:main:measure}. 
The proof of \cref{thm:main} is based on a variation of Bradshaw's technique of nested gaps for lacunary series. 
It also involves some delicate considerations about the natural numbers that can (or cannot) be represented as sums of three nonnegative cubes or as sums of four fourth powers. 
In \cref{thm:measure:theta} below we will quantify the conclusion of \cref{thm:main} by providing  a measure of linear independence for the $(\ell+1)$-tuple $(1,\th(q),\dots,\th(q)^\ell)$. 

\subsection{Notation} 
We will denote the set of nonnegative integers by $\N:=\{0,1,\ldots\}$  and by $\N_+:=\N-\{0\}$ the set of positive integers. 
The notation $\log$ will denote the natural logarithm and $\log_2$ will denote the logarithm in base 2. 

\section{Remarks on the method and comparison with the literature}
\label{sec:method}

To prove that a number is not algebraic, it is a common technique to seek for good rational approximations. 
Since $\th(q)$ is defined as a series, it is natural to approximate it by its truncations. 
However their relatively slow rate of convergence implies only that $\th(q)$ is irrational at integer inputs. 
The method of Bradshaw \cite{bradshaw} improves on the above strategy when the series is ``lacunary''. It is based on the construction of ``nested gaps'' and on the following easy observation.
\begin{remark}\label{rmk:lacunary}
Let $S=\sum_{n\geq 0} s_n$ be a series for which a tail bound of the form $\abs{\sum_{n\geq N} s_n}\leq f(N)$ is given. 
Suppose that for some $K,n_0\in\N$ we have $s_{n_0+i}=0$ for all $0\leq i <K$:    
we say that the series $S$ has a \emph{gap} of length $\geq K$ at $n_0$. 
When we have such a gap, the bound for the tail at $n_0$ can be improved to $\abs{\sum_{n\geq n_0} s_n}\leq f(n_0+K)$. 
\end{remark}
By applying this method to the (lacunary!) series representation of $\th(q)^{\ell-1}$ Bradshaw was able to show \cite[Theorem 2.0.1]{bradshaw}, for all integer $\ell$, that $\th(q)$ is not an algebraic number of degree $<\ell$. 
To extend the non-algebraicity of $\th(q)$ up to degree $=\ell$ one faces technical difficulties related to Waring's problem (I thank Martin Rivard-Cooke for pointing this to me). 
More precisely, we need the existence of arbitrarily long sequences of consecutive integers none of which is a sum of $\ell$ nonnegative $\ell$-th powers. 
This result was recently proved by the author \cite[Thm. 1.1, 1.2, 8.8]{ghilu} for $\ell\in\{3,4\}$ and is open for $\ell\geq 5$.
The aim of this article is to check that this, together with the consideration of suitable ``mild'' gaps (see \cref{sec:principle}), is enough for the proof of \cref{thm:main}. 
As a side note, we would like to remark that our lower bound for the size of gaps between sums of fourth powers, although growing to infinity, it does so very slowly.  
Therefore, it came with some surprise that these estimates are in fact good enough to have arithmetic consequences on the biquadratic theta series.

In the literature variants of the above series have been considered.  The irrationality and nonquadraticity of classical theta values $\theta_2(q)$ were studied by Duverney \cite{duverney1993,duverney}. 
Irrationality and irrationality measures of similar numbers have been considered in many works, such as \cite{irrationalityS, irrationalityBS, irrationalityMM}. 
B\'ezivin \cite{bezivin} proved the nonquadraticity of values of the more general Tschakaloff function $T_q(z) = \sum_{n=0}^\infty z^n q^{-n(n-1)/2}$. The results of B\'ezivin have been simplified by Bradshaw \cite[Chapter 3]{bradshaw} and extended by some authors \cite{morenonquadraticity}.
%As we already remaWith his method, Bradshaw showed that $\theta_k(q):=\sum_{n=0}^\infty q^{-n^k}$ is not algebraic of degree $<k$, for all integers $k,q\geq 2$ \cite[Theorem 2.0.1]{bradshaw}. 
Last but not least, a celebrated result of Nesterenko \cite{nesterenko} implies that $\theta_2(q)$ is transcendental for all nonzero algebraic $q$ satisfying $\abs{q}>1$ \cite[Theorem 4]{theta}. His proof relies on an appropriate multiplicity estimate and it exploits the differential Ramanujan identities between the quasi-modular functions $E_2(q)$, $E_4(q)$ and $E_6(q)$.

\section{A nested gaps principle for linear independence}
\label{sec:principle}

In \cref{sec:method} we mentioned that Bradshaw \cite{bradshaw} took advantage of sufficiently large gaps for the series representation of $\th(q)^{\ell-1}$, and that he applied a certain ``nested gaps'' argument to prove his results. We are going to reproduce a variation of his technique by considering the series representation of $\th(q)^{\ell}$, for $\ell\in\{3,4\}$, and by considering only those ``gaps'' that are followed by coefficients with controlled size. We call these gaps ``mild'' in \cref{def:gap} below. 
Although we are ultimately interested in (non)-algebraicity properties of $\th(q)$, a careful inspection reveals that Bradshaw's method is more naturally seen as a lemma for \emph{linear independence} of lacunary series. 
We think it is worthwile to recast Bradshaw's technique in this setting. 
However, we will not try to enunciate a criterion valid in maximal generality, in order not to obfuscate the underlying idea.  We need a few definitions. 

\begin{definition}
	%A sequence $a_0,a_1,\ldots$ of integers has \emph{polynomial growth} if $\abs{a_n}\leq c n^d$ for some $c,d>0$. 
	We define a \emph{$\tfrac 1 2$-function} to be a powerseries $f(z)=\sum_{n\in\N} a_n z^n$ with integer coefficients that is absolutely convergent for all $\abs{z}\leq 1/2$. %whose coefficients form a sequence of integers with polynomial growth, i.e.  $\abs{a_n}\leq c n^d$ for some $c,d>0$ and all $n\geq 1$.
\end{definition}

%We say that a $\tfrac 1 2$-function $f(z)=\sum_{n\in\N} a_n z^n$  is lacunary if for all $K\in\N$ there is $n\in\N$ with $a_{n+1}=\dots=a_{n+K}=0$. 
%In it has arbitrarily large gaps, eis lacunary Let  be We formalize the notion of a ``gap'' in the series representation of a 
In particular, a $\tfrac 1 2$-function can be evaluated at reciprocals of integers $q\geq 2$.

\begin{definition}\label{def:gap}
Let $f(z)=\sum_{n\in\N} a_n z^n$ be a $\tfrac 1 2$-function, let $K\in\N_+$ and $E>0$.  
We say that an index $n\in\N$ is a \emph{mild gap point} for $f(z)$, with gap-length $\geq K$ and $K$-tail-norm $\leq E$, if $a_{n+k} = 0$ for all $0\leq k< K$ and  
\[
	\sum_{i=0}^\infty\abs{a_{n+K+i}} 2^{-i} \leq  E. 
\] 
We denote by $\mild(f(z);K,E)$ the set of such mild gap points for $f$.
\end{definition}

The next theorem is the promised criterion, abstracted from Bradshaw's method, for $\Q$-linear independence of the values $f(1/q)$, $g(1/q)$ of two lacunary $\tfrac 1 2$-functions at the reciprocal of an integer. 
It essentially states that the linear independence necessarily occurs when \emph{pairs of (large enough) mild gaps of $f$ can be found inside one (larger) gap of $g$}. 
As Damien Roy pointed out to me, the proof also yields a measure of linear independence between $f(1/q)$ and $g(1/q)$. 
We explore this quantitative refinement in \cref{sec:measure}. 

\begin{theorem}[Nested Gaps Principle]\label{thm:principle}
	 %with $\abs{a_n},\abs{b_n}\leq c n^d$ for all $n\in\N_+$ and some $c,d>0$. 
Let $q\geq 2$ be an integer and let $f(z)=\sum_{n\in\N} a_n z^n$ and $g(z)=\sum_{n\in\N} b_n z^n$ be $\tfrac 1 2$-functions.  
Suppose that for every $H>0$ there are positive integers $K_1\leq K_2<K'\in\N_+$, indices $n'\leq n_1<n_2\in\N$ and real numbers $E,E'>0$ such that:
\begin{enumerate}[(i)]
\item $n_1+K_1< n_2$ and $n_2+K_2\leq n'+K'$;\label{hp:1}
\item $n_1,n_2\in \mild(f(z);K_1, E)$ and $n'\in\mild(g(z); K', E')$;\label{hp:2}
\item $\sum_{n=n_1}^{n_2-1} a_n q^{-n} \neq 0$;\label{hp:3}
\item $q^{K_1}>H E$ and $q^{K_2}>H E'$.\label{hp:4}
\end{enumerate}
Then either $g(1/q)=0$ or $f(1/q)$ and $g(1/q)$ are linearly independent over $\Q$. 
\end{theorem}
\begin{proof}
Suppose the contrary. Then there exist integers $\alpha,\beta$ such that $\alpha \neq 0 $ and 
\begin{equation}\label{principle:start}
	0 = \alpha f(1/q) + \beta g(1/q) = \sum_{n\in\N} \frac{R(n)}{q^n},
\end{equation}
where  $R(n):=\alpha a_n + \beta b_n$. 
Let $H=\max\{\abs{\alpha},\abs{\beta}\}$, then choose $K_1,K_2,K',E,E'$ and $n_1,n_2,n'$ as above. 
Now pick $i\in\{1,2\}$ arbitrarily. 
By hypothesis (\cref{hp:2}) and since $q\geq 2$ we have 
\begin{equation}\label{principle:tail}
\abs{\sum_{n= n_i}^\infty \frac{R(n)}{q^n}}
%\leq 
%\left[3 H N^3 \left(\frac q {q-1}\right)^3\frac 1 q\right]\cdot \frac 1 {q^{n_i+K}}
\leq 
\abs{\alpha}\frac{E}{q^{n_i+K_1}} + \abs{\beta}\frac{E'}{q^{n'+K'}}. 
\end{equation} 
From the estimates (\cref{hp:4}), \cref{principle:start} and $n_i+K_2\leq n'+K'$, we deduce that 
\begin{equation}\label{principle:truncated}
\abs{\sum_{n=0}^{n_i-1} \frac{R(n)}{q^n}}
<
\frac 2 {q^{n_i}}. 
\end{equation}
However, the left-hand side of \cref{principle:truncated} is a rational number with denominator at most $q^{n_i-1}$ and so it must be equal to zero. 
Having concluded this for both $n_1$ and $n_2$, we deduce that 
\[
0= \sum_{n=n_1}^{n_2-1} \frac{R(n)}{q^n} = \alpha \sum_{n=n_1}^{n_2-1} a_n q^{-n}, 
\]
against hypothesis (\cref{hp:3}).
\end{proof}

\section{Simple tail bounds}
\label{sec:tail}

In this section we present a pair of lemmas to estimate the ``tail-norms'' of a $\tfrac 1 2$-function when suitable bounds are known for its coefficients (see \cref{def:gap}). 
%Since these results are completely elementary, we will only state the simplest versions useful for our needs. 

\begin{lemma}\label{lemma:tail}
	Let $(a_n)_{n\in\N}$ be a sequence of numbers with $\abs{a_n}\leq c (n+1)$ for all $n\in\N$ and some $c>0$. 
	Then for every $n_0\in\N_+$ we have % satisfies $n_0\leq N$ for some $N\geq 3$, then
	\begin{equation*}
		\sum_{i=0}^\infty \abs {a_{n_0+i}} 2^{-i}
		\leq 
		8 c n_0.
	\end{equation*} 
\end{lemma}

\begin{proof}
	The positive function $\psi(x)=x2^{-x}$ satisfies $\psi(x)\geq \psi(2)$ for all $x\in[1,2]$ and it is monotone decreasing for $x>1/\log 2= 1.44269\ldots$, hence
	\begin{equation}\label{tail:1}
		\sum_{i=0}^\infty \abs {a_{n_0+i}} 2^{-i} \leq 2^{n_0+1}\sum_{n\geq n_0}\frac{ c (n+1)}{2^{n+1}} \leq c 2^{n_0+1}\int_{n_0}^\infty \frac{t}{2^t} dt.
	\end{equation}
	By partial integration we obtain
	\begin{equation*}
		\int_{n_0}^\infty t 2^{-t} dt 
		= 
		\frac{2^{-n_0}}{\log 2}\left( n_0 + \frac 1 {\log 2}\right)\leq \left( \frac 1 {\log 2} + \frac 1 {(\log 2)^2} \right) \frac {n_0}{ 2^{n_0}} \leq 4 n_0 2^{-n_0}.
	\end{equation*}
Together with \cref{tail:1}, this gives the lemma. 
\end{proof}

\begin{lemma}\label{lemma:tail:mild}
 Let $(a_n)_{n\in\N}$ be as in \cref{lemma:tail} for some $c>0$, and let $\kappa, n_0\in\N_+$ with $n_0+\kappa\leq N$ and $\kappa\geq \log_2 N$ for some $N$.  
 Suppose that for all $0\leq i<\kappa$ and some $E\geq 8c$ we have $\abs{a_{n_0+i}}\leq (3/2)^i E$. 
 Then 
 \begin{equation*}
		\sum_{i=0}^\infty \abs {a_{n_0+i}} 2^{-i}
		\leq 
		5 E .
 \end{equation*}
\end{lemma}

\begin{proof}
  From $\abs{a_{n_0+i}}\leq (3/2)^i E$ we get
  \[
   \sum_{i=0}^{\kappa-1} \abs {a_{n_0+i}} 2^{-i} \leq \sum_{i=0}^\infty  \left(\frac 3 4 \right)^i\cdot E = 4E.
  \]
  On the other hand, by \cref{lemma:tail} and the various inequalities relating the constants, we have 
  \[
   \sum_{i=\kappa}^{\infty} \abs {a_{n_0+i}} 2^{-i} \leq \frac 1 {2^\kappa} 8 c(n_0+\kappa) \leq \frac {1} N 8cN \leq  E.
  \]
\end{proof}

\section{Linear independence of powers of $\th$}
\label{sec:powers}

Fix $\ell\in\{3,4\}$. For all $s\in\{1,\ldots,\ell\}$ and $n\in\N$ we set 
\[
r_{\ell,s}(n) = \#\{(n_1,\ld,n_s)\in\N^s:\ n_1^\ell+\dots+n_s^\ell=n\}
\]
so that for all $q>1$  
\[
\th(q)^s = \sum_{n=0}^\infty \frac{r_{\ell,s}(n)}{q^n}. 
\]
We observe that  $\th(q)^s$ is the value at $1/q$ of the $\tfrac 1 2$-function 
\begin{equation*}
	f_{\ell,s}(z):= \sum_{n=0}^\infty r_{\ell,s}(n) z^n
\end{equation*}
for all $\ell,s$. 
Therefore we may apply \cref{thm:principle} to prove the following criterion.

\begin{proposition}\label{noncubic}
	Let $q\geq 2$ be an integer. Suppose that for every $J>0$ there are $E,N>0$, integers $K_1\leq K_2\in\N_+$ and $n_1,n_2\in \mild(f_{\ell,\ell};K_1,E)$ such that:
	\begin{enumerate}[(i)]
		\item $n_1+K_1< n_2$ and $n_2+K_2\leq N$;\label{hp:KN}
		%\item  \label{hp:mild}
		\item $r_{\ell,\ell-1}(n)=0$ for all $n_1\leq n< n_2+K_2$;\label{hp:r12}
		\item there exists $n_3\in[n_1,n_2)$ with $r_{\ell,\ell}(n_3)>0$;\label{hp:r3}
		\item $q^{K_1}>J E$ and $q^{K_2}> J N$.\label{hp1} %2^{\ell+3}J
	\end{enumerate}
	Then either $\th(q)$ is transcendental or it is algebraic with degree at least $\ell+1$.
\end{proposition}

\begin{proof}
	Suppose that $\th(q)$ is algebraic of degree at most $\ell$. Then there exist integers $\alpha_0,\dots,\alpha_\ell$ with $\alpha_\ell\neq 0$ such that 
	\begin{equation}\label{noncubic:start}
		\alpha_0 + \alpha_1 \th(q) + \dots + \alpha_\ell \th(q)^\ell = 0 .
	\end{equation}
	We define $f(z):= f_{\ell,\ell}(z)$ and 
	\begin{equation*}
		g(z):= \alpha_0 + \alpha_1 f_{\ell,1}(z) + \dots + \alpha_{\ell-1} f_{\ell,\ell-1}(z).
	\end{equation*}
	We notice that for all $s\leq \ell$ and all $n\in\N$ we have the (loose) estimate 
	\begin{equation}\label{r:loose}
		0\leq r_{\ell,s}(n) \leq (\sqrt[\ell]{n}+1)^s \leq 2^\ell (n+1). 
	\end{equation}
	%Therefore we see that $f(z)$ and $g(z)$ are $\tfrac 1 2$-functions whose coefficients have at most linear growth ($d=1$). 
	In particular for all $n\in\N$ the $n$-th coefficient of $g(z)$ has absolute value $\leq c(n+1)$ where $c=\ell \cdot 2^\ell\cdot \max\{\abs{\alpha_i}:i<\ell\}$. 
	We also notice that for $n_1\leq n<n_2+K_2$ the condition (\cref{hp:r12}) implies that $r_{\ell,s}(n)=0$ for all $s\less \ell$, i.e. that the $n$-th coefficient of $g(z)$ vanishes. 
	By \cref{lemma:tail}, this means that $n_1\in\mild(g;K',E')$, where $K'=n_2-n_1+K_2$ and $E'=8c N$. 
	Moreover (\cref{hp:r3})  is equivalent to  $\sum_{n=n_1}^{n_2-1} r_{\ell,\ell}(n) q^{-n} \neq 0$ because $r_{\ell,\ell}$ is nonnegative. 
	Thus, for any $H>0$, the hypotheses of the current proposition for $J=8 c H$ imply those of \cref{thm:principle} with $n'=n_1$ and $E'=8 c N$.  
	By \cref{noncubic:start} the numbers $f(1/q),\ g(1/q)$ are linearly dependent. 
	But since $ f(1/q)> 0$ and $\alpha_\ell\neq 0$ we also have $g(1/q)\neq 0$, so we arrive at a contradiction. 
\end{proof}

\section{Sums of powers modulo M and existence of mild gaps}
\label{sec:modM}

By the previous proposition, \cref{thm:main} is reduced to the problem of finding suitable mild gaps of $f_{\ell,\ell}$. 
In this section we present a proposition that provides ``many'' mild gaps of a prescribed type. 
This result is proved via an elementary technique known as the Maier matrix method \cite{maier}. 
We require the following definition: for every $m\in\Z$ and $M\in\N_+$ let 
\[
r_{\ell,\ell}(m,M):= \{(x_1,\dots,x_\ell)\in(\Z/M\Z)^\ell:\ x_1^\ell+\dots+x_\ell^\ell\equiv m\bmod M\}.
\]

\begin{proposition}\label{maier}
Let $K,M,m\in\N$ with $m+K<M$.  
Now let $\e_0,\ld,\e_K>0$ such that $r_{\ell,\ell}(m+k,M)\leq \e_k M^{\ell-1}$ for all $0\leq k\leq K$ and let $E_0,\ld,E_K\in\N$ such that $\alpha<1$, where
\[
\alpha:=\frac{\e_0}{E_0+1}+\ld+\frac{\e_K}{E_K+1}.
\]
Then for each $N>0$ with $N\geq M^\ell$ we have 
\[
\#\left\{
n\in[0,N-K)
\ \left|\ 
\begin{matrix}
n\equiv m\pmod M \\
r_{\ell,\ell}(n+k)\leq E_k\text{ for all $0\leq k\leq K$}
\end{matrix}
\right.
\right\}
\geq
\frac{1-\alpha}{2^\ell}\frac N M.
\]
\end{proposition}
\begin{proof}
Let $L\in\N$ such that $L^\ell M^\ell\leq N<(L+1)^\ell M^\ell$ and let $I=L^\ell M^{\ell-1}$. 
It is not difficult \cite[Prop. 8.4]{ghilu} to show that for all $0\leq k\leq K$ we have
\begin{equation*}%\label{modM}
	\sum_{i=0}^{I-1} r_{\ell,\ell}(m+k+iM) \leq L^\ell r_{\ell,\ell}(m+k,M).
\end{equation*}
From this we deduce that
\[
\#\{i\in[0,I):\ r_{\ell,\ell}(m+k+iM)> E_k\}\leq \frac{\e_k I}{E_k+1}.
\]
Therefore the $0\leq i<I$ such that $r_{\ell,\ell}(m+k+iM)\leq E_k$ for all $0\leq k\leq K$ are at least 
\[
(1-\alpha)L^\ell M^{\ell-1} = (1-\alpha) \left(\frac{L}{L+1}\right)^\ell \frac{(L+1)^\ell M^\ell}{M}\geq \frac{1-\alpha}{2^\ell}\frac N M.
\]
The proposition follows because for each such $i$ we have $m+iM<N-K$.
\end{proof}

\section{Key results from Waring's problem}

In order to find mild gap points with gap-length $K_1$  using \cref{maier} %we will need a special case of \cref{maier} where $K_1\leq K_2=K$, $E_i=0$ for $i < K_1$ and  $E_i=\tfrac 1 5 E\ 1.5^{i-K_1}$ for all $i \geq K_1$ and some $E>0$.  
it is crucial that we make $r_{\ell,\ell}(m+k,M)$ as small as possible for $k< K_1$ and that we can estimate it from above for larger values of $k$. 
%We state the required lemmas separately for $\ell=3$ and $\ell=4$.  

\begin{lemma}\label{epsilon}
Let $\ell\in\{3,4\}$ and define the following auxiliary functions of $T$ 
\begin{align*}
\kappa_3(T)&:= \frac {\sqrt T}{(\log T)^2}
& 
\kappa_4(T)&:= \frac {\log \log T}{\log \log \log T}
\\
\Xi_3(T)&:= \log \log T
&
\Xi_4(T)&:= 1.
\end{align*}
For each large enough $T$ there are natural numbers $M,m,K_1$, with $\max\{2m,4K_1\}< M$ and $M$ even, and positive constants $C_0,C_1,C_2,C_3$ such that:
\begin{enumerate}[(i)]
\item $C_0 T\leq \log M \leq C_1 T$;\label{e:M}
%\item ; \label{e:mK}
\item $K_1\geq C_2\cdot \kappa_\ell(T)$; \label{e:K}
\item $r_{\ell,\ell}(m+k,M)\leq \tfrac 1 {2 K_1}\cdot M^{\ell-1}$ for all $0\leq k < K_1$; \label{e:some} 
\item $r_{\ell,\ell}(m',M)\leq e^{C_3\Xi_\ell(T)}\cdot M^{\ell-1}$ for all $m'\in\Z$.\label{e:all}
\end{enumerate}
\end{lemma}

%\begin{lemma}\label{epsilon:3}
%For each large enough $T$ there are natural numbers $M,m,K_1$ and positive constants $C_1,C_2,C_3$ such that:
%\begin{enumerate}[(i)]
%\item $M$ is even and $C_0 T\leq \log M \leq C_1 T$;\label{e:M}
%\item $m<\tfrac 1 2 M$ and $K_1<\tfrac 1 4 M$; \label{e:mK}
%\item $K_1\geq C_2 \sqrt T (\log T)^{-2}$; \label{e:K}
%\item $r_{3,3}(m+k,M)\leq \tfrac 1 {2 K_1} M^2$ for all $0\leq k < K_1$; \label{e:some} 
%\item $r_{3,3}(m',M)\leq M^2 (\log T)^{C_3}$ for all $m'\in\Z$.\label{e:all}
%\end{enumerate}
%\end{lemma}
%
%\begin{lemma}\label{epsilon:4}
%	For each $T$ large enough $T$ there are natural numbers $M,m,K_1$ and positive constants $C_0,C_1,C_2,C_3$ such that:
%	\begin{enumerate}[(i)]
%		\item $M$ is even and $C_0 T\leq \log M \leq C_1 T$;\label{ee:M}
%		\item $m<\tfrac 1 2 M$ and $K_1<\tfrac 1 4 M$; \label{ee:mK}
%		\item $K_1\geq C_2 \frac{\log \log  T}{\log \log \log T}$; \label{ee:K}
%		\item $r_{4,4}(m+k,M)\leq \tfrac 1 {2 K_1} M^3$ for all $0\leq k < K_1$; \label{ee:some} 
%		\item $r_{4,4}(m',M)\leq C_3 M^3$ for all $m'\in\Z$.\label{ee:all}
%	\end{enumerate}
%\end{lemma}

\begin{proof}
 We are going to follow the arguments of \cite[Sec.~8]{ghilu} applied to the diagonal form $F=x_1^\ell +\dots + x_\ell^\ell$.
 %them are almost proved in \cite[Sec.8.4]{ghilu} and \cite[eq.(8.4)]{ghilu}, so here we only sketch the relevant  arguments of the referenced paper and provide the missing part of the proof.  
 In \cite[Prop. 8.1]{ghilu} we construct a set $\P_F$ with positive density in the set of all primes $p\equiv 1\pmod \ell$. 
 By the Prime Number Theorem the product $M_T:=\prod_{p} p$ of the primes $p\in \P_F\cap [1,T]$ satisfies $T \ll \log M_T \ll T$.   
 In \cite[Sec. 8.4]{ghilu} we prove that there exist two natural numbers $m,K_1<M_T$ that fulfill condition (\cref{e:some}) provided that $T\geq \tau_\ell(\gamma_\ell,K_1)$, where \cite[Def.~8.6]{ghilu} 
\begin{align*}
\tau_3(\gamma,K)&:= \gamma K^2 (\log K)^4
\\
\tau_4(\gamma,K)&:= \exp(\exp(\gamma K \log K))
\end{align*}
and $\gamma_3,\gamma_4>0$ are some absolute constants. 
 If $T$ is large enough, we may take $K_1$ so that $C_2 \kappa_\ell(T)\leq K_1<\frac 1 2 M_T$ for some small enough $C_2$.  
 By \cite[Prop. 8.2]{ghilu} with $\P_1=\emptyset$ and $\P_2=\P_F\cap[1,T]$ we have that for all $m'\in\Z$ the inequality $r_{\ell,\ell}(m',M_T)\leq \xi M_T^{\ell-1}$ holds with $\xi>0$ given by  
 \[
 \log \xi = (\ell-1)^\ell \sum_{p\in \P_F\cap [1,T]} p^{-(\ell-1)/2}. 
 \]
 The sum to the right is again estimated via the Prime Number Theorem (see \cite[Lemma 5.4]{ghilu}): if $\ell=3$ this sum is $\ll \log \log T$; if $\ell=4$ it is bounded. 
 In both cases we get the estimate $r_{\ell,\ell}(m',M)\leq e^{C_3\Xi_\ell(T)}\cdot M^{\ell-1}$ for all $m'\in\Z$ and some $C_3$. 
 Finally, we define $M:=2 M_T$. 
 All the statements in the lemma now follow because for every $m'\in \Z$ we have 
 \[
 r_{\ell,\ell}(m',M) = r_{\ell,\ell}(m',2)r_{\ell,\ell}(m',M_T) = 2^{\ell-1} r_{\ell,\ell}(m',M_T).
 \]   
\end{proof}

As we will see, the above lemma together with \cref{maier} implies that the series attached to $\th(q)^\ell$ has gaps of arbitrarily large size. 
On the other hand, we need to produce two \emph{distinct} such gaps inside a single gap attached to $\th(q)^{\ell-1}$. 
The typical gap (in $[1,N]$) between sums  of $\ell-1$ perfect $\ell$-th powers is  of size $\approx N^{1/\ell}$. 
Therefore we need to show that most gaps between sums of  $\ell$ perfect $\ell$-th powers have size $\leq N^\gamma$ for some $\gamma<1/\ell$. 
Such a result is easy to establish for $\ell=3$ with the following greedy argument. 

\begin{lemma}\label{greedy:3}
	For every $b\in\N$ there is  $n\in (b-25b^{8/27},b]$ with $r_{3,3}(n)>0$.
\end{lemma}

\begin{proof}
First notice that for every $B\in\N$ there is $x_1\in\N$ such that $x_1^3\leq B<(x_1+1)^3$. 
Such $x_1$ satisfies $B-x_1^3 \leq 6 B^{2/3}$. 
Iterating this procedure, we find in turn $x_1,x_2,x_3\in\N$ such that $0\leq (B-x_1^3)-x_2^3\leq 6 (6 B^{2/3})^{2/3}$ and
\[
0\leq B-x_1^3-x_2^3-x_3^3 \leq 6^{1+2/3+4/9} B^{8/27}<25 B^{8/27}.
\]
The lemma follows by choosing $B=b$ and $n=x_1^3+x_2^3+x_3^3$.
\end{proof}

As mentioned above, the crucial point is that $8/27<1/3$. 
The greedy argument above, for $\ell=4$, only gives $x_1,x_2,x_3,x_4$ such that 
\[B-x_1^4-x_2^4-x_3^4-x_4^4 = O(B^{(3/4)^4})\]
and $(3/4)^4 = \tfrac{5184}{16384} > 1/4$.
 %and $\tfrac{4059}{16384}< 1/4$. 
One way to overcome this problem is to prove the existence of suitable $x_1,\dots,x_4$ via the so-called ``circle method with diminishing ranges'', which might be thought as a (nontrivial) improved version of the greedy argument. 
Since the proof is technical, we perform the required computation in a separate paper \cite{ghilu:gaps4}. In that article, we extend to sums of four powers a result of Daniel for sums of three cubes \cite{daniel} and in particular we are able to show the following \cite[Corollary 1.2]{ghilu:gaps4}. 

\begin{lemma}\label{circle:4}
	For almost every $a\in \N$ there is  $n\in (a-a^{\tfrac{4059}{16384}+\varepsilon},a]$ with $r_{4,4}(n)>0$, where $\varepsilon>0$ is arbitrary.
\end{lemma}

By ``almost every $a$'' in the above lemma we mean that for every $\varepsilon>0$ and all $\delta\in(0,1)$ there is some $N_{\varepsilon,\delta}\in\N$ such that, for all $N\geq N_{\varepsilon,\delta}$ we have that the set 
\begin{equation}\label{eq:circle}
\mcl A_N:= \{ a\in [1,N]:\ r_{4,4}(n)=0 \text{ for all } n\in(a-a^{\tfrac{4059}{16384}+\varepsilon},a]\} 
\end{equation}
has cardinality $\#\mcl A_N\leq \delta N$. 

\section{Proof of Theorem 1.1}
  \label{sec:main} 
Fix $\ell\in\{3,4\}$, an integer $q\geq 2$ and an arbitrary $J>0$. 
Choose $\sigma_3\in(3,\tfrac {27} 8)$ and $\sigma_4\in(4,\tfrac{16384}{4059})$, then take $T=T(q,J,\sigma_\ell)$ large enough for the following arguments to be valid.

\subsection{Choice of parameters.}
\label{sec:main:1}

Given $T$, we choose $M,m,K_1,C_i$ as in \cref{epsilon}, then we set  $N= M^{\sigma_\ell}$ and $K_2=\tfrac 1 2 M>2K_1$.  
We also define $\xi_3 =(\log T)^{C_3}$ and $\xi_4=\max\{C_3, 32/3\}$, and finally $E=60\xi_\ell$. 
It is clear that the inequalities $q^{K_1}>JE$ and $q^{K_2}>JN$ hold if $T$ is large enough. 
In other words, condition (\cref{hp1}) of \cref{noncubic} is fulfilled.

\subsection{A set of mild gap points.}
\label{sec:main:2}

We apply \cref{maier} with $K=K_2$ and:
\begin{enumerate}
	\item $\e_k = \tfrac {1}{2K_1}$ and $E_k=0$ for $0\leq k<K_1$; \label{Ee1}
	\item $\e_{K_1+k}=\xi_\ell$ and $E_{K_1+k}= 12\xi_\ell (3/2)^{k}$ for $0\leq k\leq K_2-K_1$. \label{Ee2}
\end{enumerate}  
In addition to $m+K_2<\tfrac 1 2 M + \tfrac 1 2 M= M$ and $M^\ell < M^{\sigma_\ell}=N$, we have 
\[
\alpha:= \frac{\e_0}{E_0+1}+\dots+\frac{\e_K}{E_K+1} 
< 
K_1\frac {1}{2K_1}  
+ \sum_{k=0}^\infty \frac{\xi_\ell}{12 \xi_\ell (3/2)^{k}} = \frac{3}{4}.
\]
So \cref{maier} provides a set 
\[
\mcl B = \{b_1<b_2<\ld\}\subseteq [0, N-K_2)\cap (m+\Z M)
\] 
with cardinality $\#\mcl B\geq N/(2^{\ell+2}M)$ such that $r_{\ell,\ell}(b_i+k)\leq E_k$ for all $0\leq k\leq K_2$.  
In particular, by condition (\cref{Ee1}) above we have that all elements of $\mcl B$ are mild gap points for $f_{\ell,\ell}$ with gap-length $\geq K_1$. 
We recall from \cref{r:loose} that $r_{\ell,\ell}(n)\leq 2^\ell(n+1)$ for all $n\in \N$.
Moreover we observe that 
\[ 
12\xi_\ell\geq 8\cdot 2^\ell
\quad\text{ and }\quad
\kappa:=K_2-K_1 \geq K_1\geq \log_2 N
\]
 if $T$ is large enough. 
Therefore, by \cref{lemma:tail:mild} and condition (\cref{Ee2}), every $b_i\in\mcl B$ has $K_1$-tail-norm $\leq 5\cdot 12\xi_\ell \leq E$. 
In other words, we have $\mcl B\subseteq \mild(f_{\ell,\ell}(z);K_1,E)$.

\subsection{``Nested'' pairs of mild gaps.}
\label{sec:main:3}

We now seek to apply \cref{noncubic} to a pair of consecutive points $n_1=b_i$, $n_2=b_{i+1}$ from $\mcl B$. %mild gap points for some $i<\#\mcl B$. 
We already argued that condition (\cref{hp1}) is satisfied by our choice of parameters. Condition (\cref{hp:KN}) is fulfilled as well: $n_1+K_1< n_2$ because $b_i\equiv b_{i+1}\equiv m\pmod M$ and $K_1<M$; while $n_2+K_2<N$ because $b_{i+1}\leq \max \mcl B < N-K_2$. 
In order to fulfill condition (\cref{hp:r12}) we need to exclude any $b_i$ from the set
\[
\Bad:=\{b_i\in\mcl B:\ \exists n\in [b_i,b_{i+1}+K_2] \text{ with } r_{\ell,\ell-1}(n)\geq 1\}.
\]
Since $b_{i+1}+K_2<b_{i+1}+M\leq b_{i+2}$ for all $i\leq\#\mcl B -2$, it is clear that 
\[
\#\Bad \leq 2 \sum_{n=0}^{N} r_{\ell,\ell-1}(n),
\]
which in turn is $\leq 2(\sqrt[\ell]{N}+1)^{\ell-1} \leq 2^\ell N^{1-1/\ell}$.  
On the other hand, $\#\mcl B\geq 2^{-\ell-2} N^{1-1/\sigma_\ell}$, so $\#\Bad< (\#\mcl B)/2$ if $T$ (and so $N$) is sufficiently large. 
In particular, the complementary set $\Good:=\mcl B\setminus \Bad$ has cardinality at least $N/(2^{\ell+3} M)$. 
For every pair $(n_1,n_2)=(b_i,b_{i+1})$ with $b_i\in\Good$, condition  (\cref{hp:r12}) of \cref{noncubic} is fulfilled.

\subsection{``Separated'' pair of mild gaps.}
\label{sec:main:4}

If $\ell=3$ then every pair $(n_1,n_2)=(b_i,b_{i+1})$ with $b_i\in\mcl \Good$ satisfies condition (\cref{hp:r3}) of \cref{noncubic}. 
Indeed, recall that $n_1$ and $n_2$ are congruent (to $m$) modulo $M$, so $n_2-n_1\geq M$. 
By our choice of $\sigma_3$ we have 
\[
25 n_2^{8/27} \leq 25 N^{8/27}< N^{1/\sigma_3}=M
\]   
for every $T$ large enough, so the claim follows from \cref{greedy:3}. 
If $\ell=4$ we define $\varepsilon = \tfrac 1 2 (\sigma_4^{-1}-\tfrac {4059}{16384})$ and we consider the intervals of the form   $I(a):=(a-a^{\tfrac{4059}{16384}+\varepsilon},a]$, where $a$ is an element of the set $\mcl A\subseteq [1,N]$ given by  
\begin{equation*}
	\mcl A := \N\cap\bigcup_{b_i\in \mcl \Good} [b_i+\tfrac 1 2 M, b_i+M).
\end{equation*}

We observe that $\frac 1 2 M>N^{\tfrac{4059}{16384}+\varepsilon}$ for every $T$ large enough, so each $I(a)$ with $a\in\mcl A$ is contained in an  interval $(b_i,b_{i+1})$, for some $b_i\in\Good$. 
Suppose that no pair $(n_1,n_2)=(b_i,b_{i+1})$ with $b_i\in\mcl \Good$ satisfies condition (\cref{hp:r3}) of \cref{noncubic}. 
Then for every $a\in \mcl A$ and every $n\in I(a)$ we have $r_{4,4}(n)=0$: in other words, $\mcl A \subseteq \mcl A_N$, where $\mcl A_N$ is as in \cref{eq:circle}. 
However,
\[
\#\mcl A = \tfrac 1 2 M \cdot (\#\Good) \geq 2^{-\ell-4} N
\]
and this contradicts \cref{circle:4}, if $T$ is large enough.

\subsection{Conclusion}
\label{sec:main:5}

For every $J>0$ we proved the existence of $E,N,K_1,K_2$ and $n_1,n_2$ that meet all requirements of \cref{noncubic}.   
\Cref{thm:main} follows.

\section{Measure of linear independence} \label{sec:measure}

We present a quantitative version of the Nested Gaps Principle. %, which provides an estimate for the nonvanishing of integer linear combinations $\alpha f(1/q) + \beta g(1/q)$ if $\alpha\neq 0$. 

\begin{proposition}\label{measure:linear}
	Let $f(z)$, $g(z)$ and $q\geq 2$ be as in \cref{thm:principle}. % and let $H>0$. 
	Suppose there are positive integers $K_1\leq K_2<K'\in\N_+$, indices $n'\leq n_1<n_2\in\N$ and real numbers $E,E'>0$ meeting all conditions (\cref{hp:1})-(\cref{hp:4}) of \cref{thm:principle} for some $H>0$.  
	If $\alpha$ and $\beta$ are integers with $\alpha\neq 0$ and $\abs{\alpha}+\abs{\beta}\leq H$ then  
	\[
		\abs{ \alpha f(1/q) + \beta g(1/q) } \geq q^{-n_2}.
	\]
\end{proposition}

\begin{proof}
	We let $R(n):= \alpha a_n+ \beta b_n$ and for $i\in\{1,2\}$ we write
	\begin{equation*}
		S_i = \sum_{n=0}^{n_i-1} \frac{R(n)}{q^n} .
	\end{equation*}
	Since  $\alpha\neq 0$ we have that $S_2-S_1\neq 0$ by conditions (\cref{hp:2}) and (\cref{hp:3}). Thus, there exists $i_0\in\{1,2\}$ such that $S_{i_0}\neq 0$. 
	Since $S_{i_0}$ is a rational number with denominator $q^{-n_{i_0}+1}$, we have  $S_{i_0}\geq q^{-n_{i_0}+1}$. 
	On the other hand, as in the proof of \cref{thm:principle} we get
	\begin{align}\label{measure:tail} 
		\abs{\sum_{n=n_{i_0}}^{\infty} \frac{R(n)}{q^n} }
		\leq 
		\frac {\abs{\alpha}E} {q^{n_{i_0}+K_2}} + \frac {\abs{\beta}E'}{q^{n_{i_0}+K_1}} \leq q^{-n_{i_0}}.
	\end{align}
	Therefore
	\[
	    \alpha f(1/q)+\beta g(1/q) = S_{i_0} + \sum_{n=n_{i_0}}^\infty \frac{R(n)}{q^n}  \geq \frac {q-1}{q^{n_{i_0}}} \geq q^{-n_2}. 
	\]
	%\abs{A_1+\dots + A_5} > \tfrac 1 2 q^{-\max\{n_1,n_2\}} \geq \tfrac 1 2 q ^ {-N}. 
\end{proof}

From the above quantitative result we get the following measure of linear independence for the first powers of $\th(q)$. 

\begin{proposition}
	\label{thm:measure:theta} 
	Let $\ell\in\{3,4\}$ and $\Theta:=(1,\th(q),\dots,\th(q)^\ell)\in\R^{\ell+1}$, where $q\geq 2$ is an integer.
	Let
	$P(\mbf T)=\sum_{j=0}^\ell \alpha_j T_j$ be a nonzero linear form with integer coefficients satisfying $\abs{\alpha_j}\leq q^A$ for some $A>1.1$ and $\alpha_\ell\neq 0$. 
	 If $\ell=3$ we have  
	\begin{equation*}
		\abs{P(\Theta)} >
		\exp ( - \log q \exp ( c_3 A^2 (\log A)^2)  )
	\end{equation*}
	for some $c_3>0$, while if $\ell=4$ we have 
	\begin{equation*}
		\abs{P(\Theta)} >
		\exp ( - \log q \exp\exp\exp ( c_4 A \log A ) )
	\end{equation*}
	for some $c_4>0$. 
\end{proposition}

\begin{proof}
	%First, we assume that $\alpha_\ell\neq 0$ and 
	We wish to apply \cref{measure:linear} to the pair of $\tfrac 1 2$-functions 
	\begin{align}\label{eq:fg}
		f(z) &:= f_{\ell,\ell}(z)
		&
		g(z) &:= \sum_{j=0}^{\ell-1} \alpha_j f_{\ell,j}(z). 
	\end{align}
	We let $c=\ell 2^\ell $ and $J=8 c ( q^{A}+1)$.  
	Then we set
\begin{align*}
	T&:= c'_3 A^2 (\log A)^4 & & \text{(if $\ell=3$)} \\
	T&:= c'_3 \exp\exp (c'_4 A \log A) & & \text{(if $\ell=4$)} 
\end{align*}
	 for some $c'_\ell>0$ large enough and  we choose $K_1,K_2,N,E$ as in \cref{sec:main:1}. 
The above formula for $T$ is chosen so that the inequality 
$q^{K_1}>J E $ holds if $c'_\ell>0$ is larger than some absolute constant. 
Notice that if $c'_\ell$ is large enough we also have 
the inequality $q^{K_2}>J N $. 
Moreover, all the arguments of \cref{sec:main:2,sec:main:3,sec:main:4} are valid for every $T$ larger than some $T_0$ independent of $q$ and $J$. 
%In particular for the existence of a pair $(n_1,n_2)$ satisfying ``\cref{hp:r3}'' when $\ell=4$ one may use the quantitative refiment to \cref{greedy}.(ii) given by \cite[Theorem 1.1]{ghilu:gaps4}. 
	%Moreover also the inequalities $q^{\tfrac 3 4 K_1}\geq 8\xi_\ell K_2$ and $q^{\tfrac 1 2 K_2}\geq N$ are valid for $T$ larger than some absolute constant. 
	%Then the verifications in \cref{sec:main:5} just amount to 
%	\begin{equation}\label{eq:amount}
%	\tfrac 1 4 \log (q) K_1 > \log J,
	%\end{equation}
	%because $K_2\geq K_1$. 
	In particular, if $c'_\ell$ is large enough, there are some $n_1, n_2$ such that all the itemized conditions of \cref{noncubic} are fulfilled with this choice of $J,K_1,K_2,N,E$. 
	As in the proof of \cref{noncubic} we then see that the hypotheses of \cref{measure:linear} are fulfilled, with $n'=n_1$,  $E'=8 c N$, $H=J/(8c)$ and $f(z),g(z)$ as in \cref{eq:fg}. 
	Since $n_2 <N$ and $\log N = O(T)$, we get from \cref{measure:linear} the required estimate for 
$	P(\Theta) = \alpha_\ell f(1/q) + g(1/q)$, for some $c_\ell>0$. 
\end{proof}

Notice that the hypothesis $\alpha_\ell\neq 0$ on $P(\mbf T)$ is not restrictive. 
In fact, if $\alpha_{\ell+1-h}=\dots = \alpha_\ell = 0$ for some $h\geq 1$, we have that $P(\Theta) = \th(q)^{-h}P'(\Theta)$ where $P'(\mbf T) = \sum_{j=h}^\ell \alpha_{j-h} T_j$. 
We notice that $\th(q)\leq \th(2)\leq 2$ and so $\abs {P(\Theta)}\geq 2^{-\ell}\abs{P'(\Theta)}$. 
Therefore the estimates of \cref{thm:measure:theta} still hold if we replace $c_\ell$ by some larger absolute constant. 
However, we remark that in this situation one could apply \cref{measure:linear} to the pair of $\tfrac 1 2$-functions
\begin{align*}
		f(z) &:= f_{\ell,\ell-h}(z)
		&
		g(z) &:= \sum_{j=0}^{\ell-h-1} \alpha_j f_{\ell,j}(z). 
	\end{align*}
and obtain a measure of linear independendence that is single-exponential in ``$A$'' (as opposed to \cref{thm:measure:theta}, where the estimate is  doubly or quadruply exponential).

\section*{Acknowledgements}

I would like to thank my supervisor Damien Roy for his encouragement and for his many comments on this work, especially the suggestion of computing a quantitative measure of linear independence. 
I am grateful to Martin Rivard-Cooke for introducing me to the problem of noncubicity of cubic theta values and for mentioning the need of new results in Waring's problem for cubes. 
This work was supported in part by a full International Scholarship from the Faculty of Graduate and Postdoctoral Studies of the University of Ottawa and by NSERC. 

\bibliographystyle{plain}
\bibliography{biblio_noncubic_theta}

\end{document}